\documentclass{amsart}
\usepackage{amsfonts}
\usepackage{graphicx}
\usepackage{tabularx}
\usepackage{array}
\usepackage[usenames,dvipsnames]{color}
\usepackage{comment}
\usepackage{amsmath}
\usepackage{amsthm}
\usepackage{amssymb}
\usepackage{fullpage}
\usepackage{xcolor}
\usepackage{tikz}
\usepackage{algpseudocode}
\usepackage{algorithm}

\usepackage{enumitem}

\makeatletter
\newcommand{\multiline}[1]{%
  \begin{tabularx}{\dimexpr\linewidth-\ALG@thistlm}[t]{@{}X@{}}
    #1
  \end{tabularx}
}
\makeatother

\tikzstyle{legend_general}=[rectangle, rounded corners, thin,
                          top color= white,bottom color=lavander!25,
                          minimum width=2.5cm, minimum height=0.8cm,
                          violet]

\def\E{\ensuremath\mathbb{E}}

\newtheorem{theorem}{Theorem}[section]

\newtheorem{claim}[theorem]{Claim}
\newtheorem{lemma}[theorem]{Lemma}

\theoremstyle{definition}
\newtheorem{definition}[theorem]{Definition}

\def\epsilon{\varepsilon}

\title{Rainbow spanning trees in random subgraphs of dense regular graphs}

\author{Peter Bradshaw}
\address{Department of Mathematics, Simon Fraser University, Burnaby, Canada}
\email{pabradsh@sfu.ca}

\begin{document}
%https://link.springer.com/chapter/10.1007/978-3-540-95995-3_4
\begin{abstract}
We consider the following random model for edge-colored graphs. A graph $G$ on $n$ vertices is fixed, and a random subgraph $G_p$ is chosen by letting each edge of $G$ remain independently with probability $p$.
Then, each edge of $G_p$ is colored uniformly at random from the set $[n-1]$. 
A result of Frieze and McKay (Random Structures and Algorithms, 1994) implies that when $G = K_n$ and $p = (2 + \epsilon) \frac{\log n}{n}$ for some constant $\epsilon > 0$, then $G_p$ almost surely contains a rainbow spanning tree. In this paper, we show that if $G$ is a $d$-regular $\Omega(n)$-edge-connected graph, then when $p = (2 + \epsilon) \frac {\log n}{d}$ for some constant $\epsilon > 0$, $G_p$ almost surely contains a rainbow spanning tree. Our main tool is a new edge-replacement method for rainbow forests.
\end{abstract}
\maketitle

\section{Introduction}
Given a graph $G$ with a not necessarily proper edge-coloring $\phi:E(G) \rightarrow S$, and given a subgraph $H$ of $G$, we say that $H$ is a \emph{rainbow subgraph} of $G$ 
if $\phi$ assigns a distinct color to each edge of $H$. If $H$ is a rainbow subgraph of $G$ and is also a spanning tree of $G$, then $H$ is called a \emph{rainbow spanning tree}.
The problem of finding a rainbow spanning tree is a central problem in the theory of edge-colored graphs. Alon, Brualdi, and Shader \cite{ABS} proved that in any edge coloring of $K_n$ in which each color induces a complete bipartite graph, there exists a rainbow spanning tree. 
Furthermore, Brualdi and Hollingsworth \cite{Brualdi} proved that if the complete graph $G = K_{2n}$ is properly edge-colored with $2n-1$ colors, 
then  $G$ must admit two disjoint rainbow spanning trees, and they 
conjectured furthermore
that $E(G)$ can be partitioned into $n$ rainbow spanning trees.
Glock, K\"uhn,  Montgomery, and Osthus \cite{GKMO} recently proved that this conjecture holds when $n$ is sufficiently large, and they proved further that the $n$ rainbow spanning trees partitioning $E(G)$ can all be chosen to be isomorphic.
In the broader setting of general edge-colored graphs, Schrijver \cite{Schrijver} and Suzuki \cite{SuzukiGnC} independently proved a necessary and sufficient condition for the existence of a rainbow spanning tree in an edge-colored graph $G$, namely that $G$ contains a rainbow spanning tree if and only if the removal of all edges of any $k$ ($0 \leq k \leq n-2$) colors leaves at most $k+1$ components of $G$.

One particular graph class in which the existence of rainbow spanning trees has been extensively studied is the class of randomly constructed and randomly edge-colored graphs. Random graph constructions usually rely fully or partially on the model $G(n,p)$, which is defined as follows. We begin with a set $X$ of $n$ vertices and a value $0 < p < 1$, which may depend on $n$. Then, we say that a graph $G$ is randomly constructed from $G(n,p)$ if $G$ contains each potential edge of ${X \choose 2}$ independently with probability $p$. If a graph $G$ constructed from $G(n,p)$ satisfies some property $P$ with probability approaching $1$ as $n$ approaches infinity, then we say that $G$ satisfies $P$ \emph{asymptotically almost surely}, or \emph{a.a.s.}~for short.
Frieze and McKay \cite{FriezeMcKay} proved that if $G$ is randomly constructed from $G(n,p)$ using a value $p = (2 + \epsilon) \frac{ \log n}{n}$ with $\epsilon > 0$, and if each edge of $G$ is colored uniformly at random from the set $[n-1]$, then $G$ a.a.s.~contains a rainbow spanning tree. The coefficient of $2 + \epsilon$ in $p$ is close to tight, as when $p = (2 - \epsilon) \frac{ \log n}{n}$, $G$ is a.a.s.~missing at least one color from $[n-1]$. This probability threshold for the existence of a rainbow spanning tree
in a randomly colored graph of $G(n,p)$
may be compared to the probability threshold for the existence of an uncolored spanning tree in a random graph of $G(n,p)$; Erd\H{o}s and R\'enyi \cite{ER} proved that a random graph of $G(n,p)$ a.a.s.~contains a spanning tree when $p = (1 + \epsilon) \frac{\log n}{n}$ and a.a.s.~does not contain a spanning tree when $p = (1 - \epsilon) \frac{\log n}{n}$.

Aigner-Horev, Hefetz, and Lahiri \cite{AHL} showed additionally that certain random \emph{perturbations} of dense graphs---that is, the unions of prescribed dense graphs and sparse random graphs---often contain rainbow spanning trees when their edges are randomly colored. In particular, they proved that if $G$ is a graph on $n$ vertices with minimum degree $\delta n$ for some $\delta > 0$, and if $p = \omega(n^{-2})$, then the graph $G'$ obtained by taking the union of $G$ and a randomly constructed graph from $G(n,p)$ almost surely contains a rainbow spanning tree when each edge of $G'$ is colored uniformly at random from $[n-1]$. Furthermore, they showed that when $p = \omega(n^{-1})$, if $G'$ is obtained by taking the union of the same graph $G$ and a random graph of $G(n,p)$, then after each edge of $G'$ is randomly colored from the set $[(1 + \alpha) n]$ for some constant $\alpha > 0$, $G'$ a.a.s.~contains a rainbow spanning tree isomorphic to any prescribed tree $T$ of bounded degree on $n$ vertices.

In this paper, we will 
similarly consider rainbow spanning trees in graphs that are randomly constructed from dense graphs.
However, rather than considering random perturbations of dense graphs, we will consider random subgraphs of dense graphs. Namely, we will show that under certain conditions, if we take a random subgraph of a dense regular graph and then randomly color the edges of the resulting graph, then the edge-colored graph that we obtain a.a.s.~contains a rainbow  spanning tree. We will use the following model. Given a graph $G$ on $n$ vertices and a value $0 < p < 1$, we define $G_p$ as a random graph obtained by keeping each edge of $G$ with probability $p$; in other words, $G_p$ is a graph on $V(G)$ whose edge set is the union $E(G) \cap G^*$ for some random graph $G^*$ on $V(G)$ constructed from $G(n,p)$. If $G = K_n$, then $G_p$ is simply a random graph constructed from $G(n,p)$. Chung, Horn, and Lu \cite{CHL} study the emergence of a giant component in this random subgraph model, and Horn 
\cite{HornThesis}
dedicates an entire thesis to the study of relationships between a host graph $G$ and a random subgraph $G_p$. In fact, this random subgraph model frequently appears in stochastic theory under the name of \emph{bond percolation}, in which case the host graph $G$ is often a lattice $\mathbb Z^d$ \cite{Barsky, Grimmett}.

The main result of our paper is the following theorem.
%We say that a function $f: \mathbb N \rightarrow \mathbb N$ satisfies $f(n) \leq \polylog n$ if there exists a fixed value $C$ such that $f(n) \leq (\log n)^C$ for all sufficiently large $n$.

\begin{theorem}
\label{thmTree}
Let $G$ be a $d$-regular $\lambda$-edge-connected graph on $n$ vertices with $d \geq \lambda = \Omega (n) $. Let $p = (2 + \epsilon) \frac{ \log n}{d}$, where $\epsilon  > 0$ is a constant. If each edge of the random subgraph $G_p$ is colored uniformly at random from the set $[n-1]$, then $G_p$ a.a.s.~contains a rainbow spanning tree.
\end{theorem}

We note that the coefficient of $2 + \epsilon$ in $p$ cannot be relaxed without placing additional requirements on $G$. Indeed, if $d = (1 - o(1))n$ and $p = \frac{2 \log n}{d}$,
then a standard argument shows that the probability of every color from $[n-1]$ appearing at least once in $G_p$ is at most $ (1 + o(1))e^{-1}$. Furthermore, for any value $d = \Omega(n)$,
a standard second-moment argument shows that $p$ cannot be reduced to $(1 - \epsilon) \frac{\log n}{d}$ without a.a.s.~introducing an isolated vertex in $G_p$.

We outline our strategy for proving Theorem \ref{thmTree} as follows. Rather than considering a single random subgraph of $G$, we will consider a family $\mathcal G$ of many randomly constructed subgraphs of $G$, and we will aim to find a rainbow spanning tree in a random coloring of the graph $G' = \bigcup \mathcal G$. Specifically, $\mathcal G$ will consist of
one subgraph $G_{p}$ which is generated by keeping each edge of $G$ with probability 
$p = (2 + \frac{1}{2}\epsilon) \frac{\log n}{d}$
and many sparse subgraphs which are generated by keeping each edge of $G$ with probability less than $\frac{\sqrt{\log n}}{n}$.
%We will show that each edge of $G$ belongs to $G'$ with probability less than $(2 + \epsilon) \frac{\log n}{d}$, so it will be sufficient to prove the theorem for $G'$.
We will show that $G_{p}$ a.a.s.~contains a rainbow forest whose edges are missing at most $\log^3 n$ colors, and then, using our random sparse subgraphs and an edge replacement argument, we will add these remaining colors to our rainbow forest, ultimately constructing a rainbow spanning tree in $G'$.

\section{Tools}
\label{secTools}
We establish some tools that we will use throughout the proof of Theorem \ref{thmTree}. First,
 we will need the well-known Markov inequality, which appears as Theorem 3.1 of \cite{Mitzenmacher}.
\begin{theorem}
\label{thmMarkov}
Let $Y$ be a nonnegative random variable. For any value $a>0$, 
$$Pr(Y \geq a ) \leq \frac{\E[Y]}{a}.$$ 
\end{theorem}
Next, 
we will use the following form of the Chernoff bound, which appears in Chapter 4 of \cite{Mitzenmacher}. 
\begin{theorem}
\label{thmChernoff}
Let $Y$ be a random variable that is the sum of pairwise independent indicator variables---that is, variables taking values in $\{0,1\}$. Let $\mu$ be the expected value of $Y$. For any value $\delta \in (0,1)$,
$$\Pr(Y < (1 - \delta) \mu) \leq \left (  \frac{e^{-\delta}}{(1-\delta)^{1-\delta}} \right )^{\mu}.$$%
%and for any value $\delta > 0$,
%$$ \Pr(Y > (1 + \delta) \mu) \leq \left (  \frac{e^{\delta}}{(1+\delta)^{1+\delta}} \right )^{\mu}.$$
\end{theorem}

Finally, we will need two lemmas about highly edge-connected graphs.
\begin{lemma}
\label{lem:cuts}
Let $G$ be a $\lambda$-edge-connected graph with $\lambda = \Omega(n)$, and let $p < \frac{\log^2 n}{n}$. Then, it a.a.s.~holds that the random subgraph $G_p$ contains at most $\frac{1}{2} |S|$ edges in each edge-cut $S$ of $G$.
\end{lemma}
\begin{proof}
Consider an edge-cut $S$ of $G$, and let $k = |S|$. The probability $P$ that $G_p$ contains at least $\frac{1}{2}k$ edges in $S$ is at most 
\begin{eqnarray*}
P \leq {k \choose {k/2}} p^{k/2} < \frac{(pk)^{k/2}}{(k/2)!} < \frac{(pk)^{k/2}}{(k/2e)^{k/2}} = \left ( 2ep \right )^{k/2} .
\end{eqnarray*}
Therefore, as $k \geq \lambda = \Omega(n)$,
\[\log P < \frac{k}{2} (2 \log \log n - \log n + \log 2 + 1) =- \Omega (n \log n).\]
This shows that $2^n P = o(1)$. Since $G$ contains only $2^n$ edge-cuts,
it follows from a union bound that the lemma holds.
\end{proof}

\begin{lemma}
\label{lemmaStraddle}
Let $G$ be a $\lambda$-edge-connected graph, and let $H$ be a subgraph of $G$ with $t \geq 2$ components.
If $\Tilde G$ is a subgraph of $G$ containing at most $\frac{1}{2}|S|$ edges in each edge-cut $S$ of $G$,
then at least $\frac{1}{4} \lambda t$ edges of $E(G) \setminus E(\Tilde G)$ join separate components of $H$.
\end{lemma}
\begin{proof}
We obtain a graph $G_H$ on $t$ vertices from $G \setminus E(\Tilde G)$ by contracting each component of $H$ to a single vertex. Since $G$ is $\lambda$-connected, and since each component of $H$ has a boundary containing at least $\frac{1}{2}\lambda$ edges not belonging to $E(\Tilde G)$,
$G_H$ has minimum degree at least $\frac{1}{2}\lambda$. Therefore, 
$G_H$ has at least $\frac{1}{4}\lambda t$ edges, each of which corresponds to an edge in $E(G) \setminus E(\Tilde G)$ joining distinct components of $H$. 
\end{proof}

\section{Proof of Theorem \ref{thmTree}}
\label{secProof}
We fix a $d$-regular graph $G$ on $n$ vertices and a value $0 < \epsilon < 1$. 
 We tacitly assume that $n$ is large, and we omit floors and ceilings wherever they do not affect our arguments.
We write $\lambda$ for the edge-connectivity of $G$, and we assume that $d \geq \lambda = \Omega (n) $. We aim to show that 
when $p' = (2 + \epsilon) \frac{\log n}{d}$ for some $\epsilon > 0$,
a random subgraph $G_{p'}$ a.a.s.~contains a rainbow spanning tree when each edge is colored uniformly at random from the color set $[n-1]$.

\subsection{Strategy outline}
\label{secStrat}
%We write $\ell(n)$ for the polylog function $\frac{n}{d}$.
We let $p = \frac{(2 + \frac{1}{2}\epsilon ) \log n}{d}$, and for $2 \leq t \leq  \lfloor \log^3 n \rfloor - 1  $, we let 
$s_t = \frac{ \sqrt{\log n} }{t n} $. 
In order to prove Theorem \ref{thmTree}, we will consider the union $G'$ of a family 
$\mathcal G = \{G_p\} \cup \{ G_{s_t}
: 2 \leq t \leq  \lfloor \log^3 n \rfloor - 1\}
$
of random subgraphs of $G$.
We assign each edge $e \in E(G')$ a color $\phi(e)$ uniformly at random from the set $[n-1]$.
Each edge of $G$ belongs to $G'$ with probability at most
\begin{eqnarray*}
p + s_2 + \dots+ s_{\lfloor \log^3 n \rfloor - 1 } &= & \left  (2 + \frac{1}{2}\epsilon \right )\frac{\log n}{d}   + \frac{ \sqrt{ \log n} }{ n  }\left ( \frac{1}{2} + \frac{1}{3} + \dots + \frac{1}{\lfloor \log^3 n \rfloor - 1} \right ) \\
 & < & \left  (2 + \frac{1}{2}\epsilon \right )\frac{\log n}{d} +    \frac{ 4 \log \log n \sqrt{ \log n}}{n}  \\
 & < & (2 + \epsilon) \frac{ \log n}{d}
.
\end{eqnarray*}
Therefore, it suffices to show that $G'$ a.a.s.~contains a rainbow spanning tree.
We write $\phi:E(G') \rightarrow [n-1]$ for the random edge-coloring function of $G'$.
We note that as $p$ is much larger than $\sum_{t=2}^{  \lfloor \log^3 n \rfloor - 1 } s_t$, most of the edges in $G'$ come from $G_p$, and comparitively few edges come from a graph $G_{s_t}$. 
By Lemma \ref{lem:cuts}, for each edge-cut $S$ of $G$, we may assume a.a.s.~that $G'$ has at most $\frac{1}{2}|S|$ edges in $S$.

Throughout our argument, we would like to
assume that not only is $G'$ randomly edge-colored, but that
the subgraphs contained in $\mathcal G$
also have edges colored uniformly at random from the set $[n-1]$ and that the edge-colored graph $G'$ is obtained as a union of the edge-colored subgraphs in $\mathcal G$.
 However, this leads to a potential problem, as an edge $e \in E(G)$ may exist in two subgraphs in $\mathcal G$ with different colors, in which case the color $\phi(e)$ that $e$ receives in $G'$ is not well defined. To avoid this problem, we
 define the color $\phi(e)$ of each edge $e \in E(G')$ by the following rule.
 First, we randomly color each edge of each graph of $\mathcal G$ uniformly at random.
Next, we order the graphs in $\mathcal G$ as $G_p,G_{s_{\lfloor \log^3 n \rfloor - 1}},\dots, G_{s_2} $.
 Then, if some edge $e \in E(G)$ 
 appears in multiple graphs of $\mathcal G$,
 we let $\phi(e)$ be determined by the graph 
 that  
 appears earliest in the ordering of $\mathcal G$
 among all graphs containing the edge $e$. In particular, for each edge $e \in E(G_p)$, $\phi(e)$ is determined by the color of $e$ in $G_p$.
 %appears in a graph $G_{p}$ with color $c_1$ and in some $G_{s_t}$ with color $c_2$, then we say $\phi(e) = c_1$---that is, $e$ receives the color from $G_p$. If $e$ appears in a graph $G_{s_{t_1}}$ with color $c_1$ and in another graph $G_{s_{t_2}}$ with color $c_2$ with $t_1 > t_2$, then we say $\phi(e) = c_1$. 
 It is easy to check that the edge-coloring function $\phi:E(G') \rightarrow [n-1]$ still colors each edge of $G'$
 uniformly at random from $[n-1]$ under these rules.
 In order to observe these rules, whenever we consider the color of an edge $e$ from a random graph $G_{s_t} \in \mathcal G \setminus \{G_p\}$, we will check that $e$ does not belong to any subgraph of $\mathcal G$ appearing before $G_{s_t}$, which will guarantee that $\phi(e)$ is determined by the random color of $e$ in $G_{s_t}$.
For each value $j \in [n-1]$, we write $E_j$ for the subset of edges of $G_p$ with color $j$.

The technique of constructing $G'$ as the union of multiple random graphs is known as \textit{multiple exposure}. The reason for using this technique is that at certain points in our argument, we need to estimate the probability that $G'$ has an edge within a specific subset of $V(G) \choose 2$. However, if we consider a certain subset of $V(G)\choose 2$ multiple times in our argument, we cannot be sure that the probabilities we estimate in a later part of our argument are independent of probabilities that we have estimated in earlier parts of our argument. Therefore, it is often convenient to use a ``fresh" random set of edges that is independent of all previous choices, so that we can guarantee that the probabilities we estimate are independent of all previously measured probabilities. A simpler form of multiple exposure is used by Fernandez de la Vega \cite{deLaVega} (explained in \cite[Chapter 8]{Bollobas}) to find
a long path in a random graph.

%As $ G'$ contains the bulk of the edges in our random graph $G$, we will need to rely mainly on $G_p$ to find our rainbow spanning tree in $G'$. It is only when we get stuck using $G_p$ that we will turn to our sparse graphs $G_{s_t}$ for assistance. 
We will use our graph family $\mathcal G$ as follows. First, we will show that
we can a.a.s.~find a rainbow forest $F_0$ in $G_p$ with at least $n - \log^3 n$ edges.
Next, we would ideally like to use our sparse graphs $G_{s_t}$ to add edges with colors of $[n-1] \setminus \phi(E(F_0))$ between components of $F_0$ in order to extend $F_0$ into a rainbow spanning tree in $G'$. However, if we attempt to do this directly, we will likely find that our families $G_{s_t}$ are too sparse and that we cannot find an edge in an available color between two components of $F_0$. Therefore, in order to extend $F_0$ to a rainbow spanning tree,
we will need to use the idea of \emph{replacing} edges in $F_0$, which we define as follows.

\begin{definition}
\label{defReplace}
Let $F$ be a forest. We say that an edge $r \in E(F)$ is \textit{replaceable} by an edge $e \in {{V(F)} \choose 2} \setminus E(F)$ if $F + e$ has a cycle containing $r$. If $S \subseteq {{V(F)} \choose 2} \setminus E(F)$, then we say that $r \in E(F)$ is replaceable by $S$ if $r$ is replaceable by some edge of $S$.
\end{definition}
%In Definition \ref{defReplace}, if $e$ shares both endpoints with $r$, then we consider $F + e$ to be a multigraph with a $2$-cycle consisting of the edges $e$ and $r$, and thus we say that $e$ replaces $r$.

We give an informal explanation of how the idea of replacing edges in $F_0$ might help us extend $F_0$ to a rainbow spanning tree in $G'$. Suppose that $F_0$ has two components, and the only color missing from $F_0$ is red. We would like to find a red edge $e^*$ in one of our graphs $G_{s_t}$ that connects the two components of $F_0$ and add $e^*$ to $F_0$ in order to extend $F_0$ to a rainbow spanning tree. However, it is possible that no such red edge $e^*$ exists. 
In this situation, we may instead consider an arbitrary
red edge $e$ in $G_p$, which a.a.s.~must exist. If $e$ joins the two components of $F_0$, then we have our rainbow spanning tree; otherwise, $F_0 + e$ contains a cycle $C$. Then, we may modify $F_0$ by adding $e$ and removing some other edge $r \in E(C)$ of a different color, say yellow. Now, we may look for a yellow edge in one of our graphs $G_{s_t}$ that connects the two components of $F_0$. If no such yellow edge exists, then we can repeat the process above, replacing an edge of $F_0$ with a yellow edge in $G_p$ and searching for edges in some $G_{s_t}$ of a third color that connect the components of $F_0$. The key idea is that if we repeat this process for enough colors, then a graph $G_{s_t}$ must eventually have an edge connecting the components of $F_0$ in a desired color. We will also be able to apply this idea when considering a forest with more than two components.

\subsection{Some properties of $G_p$}
We will establish some properties of our graph $G_p$ that we a.a.s.~can expect will hold. After establishing these properties, we will be able to make certain deterministic statements about $G_p$ without invoking probability.  
\begin{lemma}
\label{lemmaBigF}
$G_p$ a.a.s.~contains a rainbow forest with at least $n - \log^3 n$ edges.
\end{lemma}
\begin{proof}
Rather than working directly with $G_p$, we construct a random directed graph $H$ on $V(G)$ as follows.
First, for each edge $uv \in E(G)$ and each color $i \in [n-1]$, we add an arc $uv$ of color $i$ to $H$ with probability $q = \frac{p}{2n}$, and we also add an arc $vu$ of color $i$ to $H$ with probability $q$. We call the resulting digraph $H'$, and we note that $H'$ may have parallel arcs or directed $2$-cycles. We obtain $H$ from $H'$ by deleting any edge contained in a set of parallel arcs or in a directed $2$-cycle. It is easy to show that for each $uv \in E(G)$, some arc joins $u$ and $v$ in $H$ with probability at most $2nq = p$, and the probability distribution on the color of this arc is uniform. Therefore, if we can show that the lemma holds for 
$H$
after removing edge orientations, then this will also imply that the lemma holds for $G_p$.

We construct an auxiliary bipartite graph $B$. For one partite set of $B$, we use the set $[n]$, and for the other partite set of $B$, we use the set $V(G)$. For a value $j \in [n-1]$ and a vertex $v \in V(G)$, we add an edge between $j$ and $v$ if and only if there exists an arc in $H'$ of color $j$ outgoing from $v$. Then, for each vertex $v \in V(G)$, we add an edge in $B$ between $v$ and the integer $n$ (with probability $1$).

We claim that each potential edge of $V(G) \times [n]$ is added to $B$ independently with probability at least $q^*= 1 - (1 - q)^d$.
 Indeed, let $j \in [n]$, and let $v \in V(G)$.
 If $j = n$, then the statement is clear. Otherwise, in $G$, there are $d$ edges incident to $v$, and an arc of color $j$ outgoing from $v$ is added along each edge with probability $q$. Therefore, the probability that $v$ is adjacent to $j$ in $B$ is equal to 
$1 - (1 - q)^{d} = q^*$,
and thus we see that $B$ is a randomly constructed bipartite graph. By a rough estimate using the binomial theorem, 
$$q^* > qd - q^2 d^2 = qd ( 1 - qd) 
= (1 - o(1))\frac{pd}{2n} > (1 + \alpha)
 \frac{ \log n}{n},$$
for some sufficiently small constant $\alpha > 0$.
Then, by a classical theorem of Erd\H{o}s and R\'enyi \cite{ERMatrix} (originally stated for the permanents of random matrices), $B$ a.a.s.~contains a perfect matching. 

Now, since $B$ has a perfect matching, $H'$ contains a directed rainbow subgraph $D$ in which all vertices but one
have out-degree $1$, and one vertex $w$ (the vertex matched to $n$ in $B$) has out-degree $0$. Furthermore, since $B$ is
constructed without considering the tails of the arcs in $H'$, we may consider $D$ to be constructed by randomly assigning each vertex $v \in V(G) \setminus \{w\}$ a random out-neighbor from $N_G^+(v)$.
Then, we may obtain a rainbow forest on $H'$ by deleting a single edge from each cycle of $D$.
With this in mind, we estimate the number of cycles in $D$.

 If $C$ is a fixed undirected $k$-cycle in $G$, then the probability that $C$ appears in the randomly constructed digraph $D$ is at most $ \frac{2}{d^k} $. Since the number of undirected cycles of length $k \geq 3$ in $G$ is at most $\frac{nd^{k-1}}{2k}$, we estimate the expected number of cycles in $D$ as follows:
\begin{eqnarray*}
\E(\textrm{number of cycles in $D$}) & \leq & \sum_{k = 3}^{n-1} \frac{nd^{k-1}}{2k} \cdot \frac{2}{d^k} \\
 & = & \sum_{k = 3}^{n-1} \frac{n}{dk}\\
 &=& O(\log n)
\end{eqnarray*}
Therefore, by Markov's inequality (Theorem \ref{thmMarkov}), it a.a.s.~holds that $D$ has at most $\log^2 n - 1 $ cycles. Hence, if we delete an edge from each cycle of $D$, we obtain a rainbow forest $F'$ in $H'$ containing at least $n - \log^2 n$ edges.

Finally, to obtain our rainbow forest in $H$, we need to delete from $F'$ any arcs that were deleted from $H'$ to form $H$. An arc joining two vertices $u$ and $v$ in $H'$ is deleted if and only if $H$ contains two parallel edges joining $u$ and $v$ (when ignoring arc direction); therefore, the number of arcs deleted from $H'$ to form $H$ is at most twice the number of parallel edge pairs in $H'$. The probability of two vertices $u$ and $v$ being joined in $H'$ by a pair of parallel edges of color $i$ and $j$ is less than $p^2/n^2$, 
which implies that the expected number of parallel edge pairs in $H'$ is less than $p^2 n^2 = O(\log^2 n)$. Hence, by Markov's inequality, we may assume a.a.s.~that the number of parallel edge pairs in $H'$ is less than $\frac{1}{3}\log^3 n$. Therefore, at most $ \frac{2}{3} \log^3 n$ edges are deleted from $H'$ to form $H$, and hence we may delete at most $ \frac{2}{3} \log^3 n$ edges from $F'$ in order to find a rainbow forest on $H$ with at least $n - \log^2 n - \frac{2}{3} \log^3 n - 1 > n - \log^3 n$ edges. This completes the proof.
\end{proof}

The next lemma estimates the number of vertices in $G_p$ incident to an edge from a given color set $K \subseteq [n-1]$. Given a color set $K$, we expect $G'$ to contain a set $E_K$ of $|E(G)|\frac{p |K|}{n-1} > \frac{1}{2}dp|K| > |K| \log n$ edges of a color from $K$.
If $|K| \log n = o(n)$, then we expect $G'[E_K]$ to be sparse, so we also expect $\Omega ( |K| \log n )$ vertices of $G_p$ to be incident to an edge of a color from $K$. The following lemma shows us that for all sets $K \subseteq [n-1]$ that are not too large, this estimate is not far from the truth.

\begin{lemma}
Let $\omega(n)$ be an unbounded increasing function. It a.a.s.~holds that for each color set $K \subseteq [n-1]$ satisfying $| K| \leq \frac{n}{ \omega(n) \log n}$, at least $| K | \cdot \frac{\log n}{\omega(n)}$ vertices $v \in V(G_p)$ are incident in $G_p$ to at least one edge of a color from $K$.
\label{lemmaHitV}
\end{lemma}
\begin{proof}

Let $K \subseteq [n-1]$ be fixed, and let $|K| = k \leq \frac{n}{\omega(n) \log n}$. We consider the random subgraph $H$ of $G_p$ consisting of those edges of $E(G_p)$ with a color in $K$.
We note that $H$ is a random subgraph of $G$ obtained by keeping each edge $e \in E(G)$ with probability $\frac{kp}{n-1}$ and then coloring $e$ uniformly at random from the set $K$.

We orient the edges of $H$ as follows. First, we give $E(G)$ an orientation in which the in-degree of each vertex $v$ is either $\lfloor d/2 \rfloor$ or $\lceil d/2 \rceil$. This is possible when $d$ is even by the Eulerian property of $G$, and this is possible when $d$ is odd by adding a universal vertex to $G$ to obtain an Eulerian graph. Then, we let each edge $e \in E(H)$ inhert its orientation from $G$.

We would like to estimate the number of vertices in $H$ with positive in-degree. 
Since $H$ is a random subgraph of $G$, the probability that a vertex $v \in V(H)$ has in-degree $0$ is
is $(1 - \frac{pk}{n-1})^{\deg^-_G(v)}$. Therefore, the expected number $\mu$ of vertices in $H$ with positive in-degree satisfies
\[\mu =  \sum_{v \in V(G)}( 1 - (1 -\frac{pk}{n-1})^{\deg^-_G(v)})
\geq 
n( 1 - (1 -\frac{pk}{n-1})^{\lfloor d/2 \rfloor }) 
\geq  \lfloor \frac{d}{2} \rfloor pk - \frac{(dpk)^2}{n},
\]
where the last inequality follows from a generous estimate using the binomial theorem.
By substituting $p = \frac{(2 + \frac{1}{2} \epsilon) \log n}{d}$ and applying our upper bound on $k$, we see that 
\[\mu \geq \left (1 + \frac{1}{4} \epsilon - o(1) \right ) k \log n  > \left (1 + \frac{1}{8}\epsilon \right ) k \log n.\]

We would like to estimate the probability that $Y \geq \frac{k \log n}{\omega(n)}$. As $Y$ is the sum of pairwise-independent indicator variables, a Chernoff bound (Theorem \ref{thmChernoff}) may be applied to $\mu$, implying
$$ \Pr(Y < (1 - \delta) \mu) \leq \left (  \frac{e^{-\delta}}{(1-\delta)^{1-\delta}} \right )^{\mu} = \left (  \frac{e}{(1-\delta)^{\frac{\delta - 1}{\delta}}} \right )^{- \delta \mu}.$$
We let $\delta = 1 -\frac{1}{\omega(n)}$; then, since 
$\lim_{n \rightarrow \infty} (1-\delta)^{\frac{\delta - 1}{\delta}} = 1,$
we have that
\[
\Pr \left (Y <\frac{k \log n}{\omega(n)} \right )  = \Pr(Y < (1 - \delta) \mu) \leq \exp \left (-(1 - o(1))\delta \mu \right ) < \exp \left ( - \left (1 + \frac{1}{16}\epsilon \right )  k   \log n  \right ). \]

Therefore, summing over all valid subsets $K \subseteq [n-1]$, the probability that the lemma does not hold for some subset $K \subseteq [n-1]$ of appropriate size is less than
\begin{eqnarray*}
& &\sum_{k = 1}^{\infty} {{n-1} \choose k}\exp \left ( - \left (1 + \frac{1}{16}\epsilon \right )  k \log n \right )  \\
&<& \sum_{k = 1}^{\infty} \exp \left (k \log n - \left (1 + \frac{1}{16}\epsilon \right ) k \log n \right )  \\
&= &\sum_{k = 1}^{\infty} \exp \left (-\frac{1}{16}\epsilon k \log n \right ) = \frac{\exp(-\frac{1}{16}\epsilon \log n)}{1 - \exp(-\frac{1}{16}\epsilon\log n)} = o(1). \\
\end{eqnarray*}

Therefore, a.a.s., the lemma holds for all color subsets $K \subseteq [n-1]$ of appropriate size.
\end{proof}

\subsection{Growing a rainbow spanning tree in $G'$: edge replacement}
\label{secIt}
Now we will begin to construct our rainbow spanning tree in $G'$. 
Recall that for an edge $e$ in a graph $G_{s_t} \in \mathcal G \setminus \{G_p\}$,
$\phi(e)$ may not be determined by
the color of $e$ in $G_{s_t}$, namely when some other graph of $\mathcal G$ contains $e$ and appears before $G_{s_t}$ in the ordering 
$G_p,G_{s_{\lfloor \log^3 n \rfloor - 1}},\dots, G_{s_2} $.
In order to avoid problems that may arise from this detail as we construct our rainbow spanning tree in $G'$, whenever we consider a colored edge $e$ in such a graph
$G_{s_t} \in \mathcal G \setminus \{G_p\}$,
we will first check that $e$ does not belong to any graph appearing before $G_{s_t}$ in our ordering of $\mathcal G$.

We will start with a rainbow forest $F_0$ on $G_p$.
By Lemma \ref{lemmaBigF}, we may choose $F_0$ to contain at least $n - \log^3 n$ edges. As explained in Section \ref{secStrat}, we wish to extend our rainbow forest by taking edges from our sparse families $G_{s_t}$ and adding them to $F_0$. In order to ensure that we can add edges of appropriate colors, we will often need to replace edges in our rainbow forest. In the following lemma, we use an argument involving edge replacement to show that
if we begin with a rainbow forest $F^*$ in $G_p$ that is maximal in some sense,
then for any color $\sigma \in [n-1] \setminus \phi(E(F^*))$ (that is, a color $\sigma$ ``missed" by $F^*$), we can find many colors $j$ for which there exists a rainbow forest $F$ in $G_p \cup E(F^*)$ with color set $\phi(E(F^*)) \cup \{\sigma\} \setminus \{j\}$---that is, such that $F$ ``misses" the color $j$.

\begin{lemma}
Let $F^*$ be a rainbow forest on $G'$ with at most $n-2$ edges, and let $\sigma \in [n-1] \setminus \phi(E(F^*))$. Let $J \subseteq \phi(E(F^*)) \cup \{\sigma\}$ denote the set of values $j \in [n-1]$ for which there exists a rainbow forest $F$ on $G_p \cup E(F^*)$ satisfying the following properties:
\begin{itemize}
\item The components of $F$ partition $V(G)$ in the same way as the components of $F^*$.
\item $\phi(E(F)) = \phi(E(F^*)) \cup \{\sigma\} \setminus \{j\}$.
%\item For each $e \in E(F)$, either $e \in E(G_{\psi(e)})$, or $e \in E(F^*)$ and $\psi(e) = \phi (e)$.
\end{itemize}
If there exists no rainbow forest in $G_p \cup E(F^*)$ with $|E(F^*)| + 1$ edges, then
$$|J| \geq   \frac{n}{  \log \log n} .$$
\label{lemmaMany}
\end{lemma}
This second condition of Lemma \ref{lemmaMany} states
that we only consider those rainbow forests $F$ for which the colors of $F$ consist of those colors used by $F^*$ and also possibly $\sigma$.
The set $J$ is then constructed by considering every such rainbow forest $F$
and adding the unique color from
$ \phi(E(F^*)) \cup \{\sigma\}$
that is missing from $E(F)$.
Then, the conclusion of the lemma states that if we cannot construct a rainbow forest with more edges than $F^*$ using just colored edges of $F^*$ and $G_p$, then this set $J$ of colors missed by the our forests $F$ must be large.
The remainder of Section \ref{secIt} will be dedicated to proving this lemma.

We note that Lemma \ref{lemmaMany}
makes a deterministic statement about
the random graph $G_p$ without any mention of probability, which is possible by assuming Lemma \ref{lemmaHitV}.
We will later apply Lemma \ref{lemmaMany} first with $F^* = F_0$, and afterward we will apply the lemma to other rainbow forests $F^*$. 
\begin{proof}[Proof of Lemma \ref{lemmaMany}]
We fix our rainbow forest $F^*$, and 
we assume that there exists no rainbow forest on $G_p \cup E(F^*)$ with $|E(F^*)| + 1$ edges.
We let $\sigma \in [n-1] \setminus \phi(E(F^*))$. We define a countably infinite rooted tree $\mathcal{T}$ which will store information about rainbow forests that satisfy the properties listed in the lemma. Each node of $\mathcal{T}$ will store a pair $(F,j)$, 
where $F$ is a rainbow forest spanning $V(G)$, and $\{j\} = \phi(E(F^*)) \cup \{\sigma\} \setminus \phi(E(F))$. With this definition, $j$ is the unique color from $\phi(E(F^*)) \cup \{\sigma\}$ that is ``missed" by the edges of $F$. For a pair $(F,j)$ stored in a node $\nu$ of $\mathcal{T}$, we say that $j$ is the \textit{color} of $\nu$. If a node $\nu$ of $\mathcal T$ stores a pair $(F,j)$, then we will often identify $\nu$ and $(F,j)$.

We construct $\mathcal{T}$ by a recursive procedure that considers each leaf of $\mathcal{T}$ and then adds children in $\mathcal T$ to that leaf according to certain rules. The procedure is carried out as follows.
\begin{algorithm}[H]
\caption{Construct $\mathcal T$}\label{alg:cap}
\begin{algorithmic}[1] 
\State Let $\mathcal{T}$ consist of a single root node $(F^*,\sigma)$. 
\While{$1=1$}
\State \label{infLoop} For each leaf $L = (F,j)$ of maximum height in $\mathcal{T}$, execute the following steps: 
\begin{enumerate}
\setlength{\itemindent}{1.2cm}
\item \multiline{Let $R$ denote the set of edges $e \in E(F)$ replaceable by the set $E_j \subseteq E(G_p)$, which \\ consists of those edges of color $j$.}
\item For each edge $r \in R$, add a child node of $L$ storing the pair $(F + e - r, \phi(r))$. 
\end{enumerate}
\EndWhile
\end{algorithmic}
\end{algorithm}

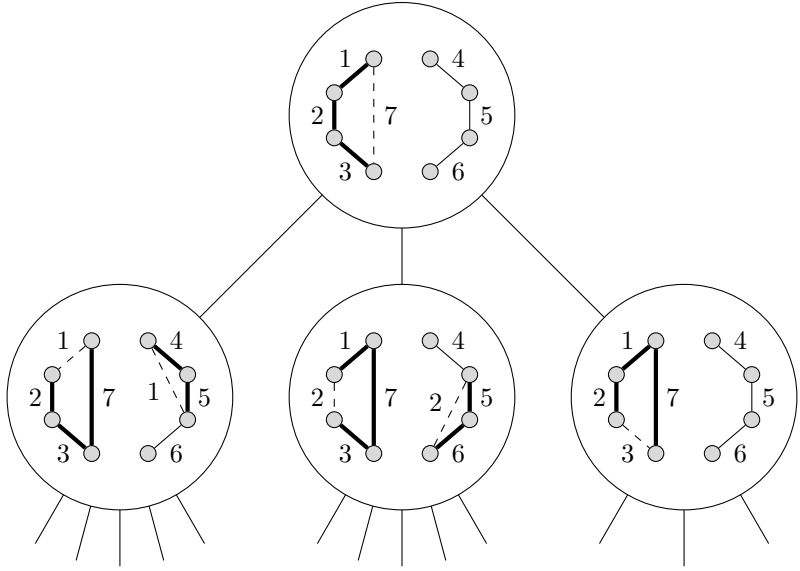
\begin{figure}
\begin{tikzpicture}
[scale=1.5,auto=left,every node/.style={circle,fill=gray!30,minimum size = 6pt,inner sep=0pt}]

\node (z) at (-0.5,5.5) [fill = white]  {$1$};
\node (z) at (-0.5,5.5-2.5) [fill = white]  {$1$};
\node (z) at (-0.5-2.5,5.5-2.5) [fill = white]  {$1$};
\node (z) at (-0.5+2.5,5.5-2.5) [fill = white]  {$1$};
\node (z) at (-0.75,5) [fill = white]  {$2$};
\node (z) at (-0.75,5-2.5) [fill = white]  {$2$};
\node (z) at (-0.75-2.5,5-2.5) [fill = white]  {$2$};
\node (z) at (-0.75+2.5,5-2.5) [fill = white]  {$2$};
\node (z) at (-0.5,4.5) [fill = white]  {$3$};
\node (z) at (-0.5,4.5-2.5) [fill = white]  {$3$};
\node (z) at (-0.5-2.5,4.5-2.5) [fill = white]  {$3$};
\node (z) at (-0.5+2.5,4.5-2.5) [fill = white]  {$3$};
\node (z) at (-0.5,4.5) [fill = white]  {$3$};
\node (z) at (-0.5,4.5) [fill = white]  {$3$};

\node (z) at (0.5,5.5) [fill = white]  {$4$};
\node (z) at (0.5,3) [fill = white]  {$4$};
\node (z) at (3,3) [fill = white]  {$4$};
\node (z) at (-2,3) [fill = white]  {$4$};

\node (z) at (0.75,5) [fill = white]  {$5$};
\node (z) at (0.75,2.5) [fill = white]  {$5$};
\node (z) at (0.75-2.5,2.5) [fill = white]  {$5$};
\node (z) at (0.75+2.5,2.5) [fill = white]  {$5$};

\node (z) at (0.5,4.5) [fill = white]  {$6$};
\node (z) at (0.5,2) [fill = white]  {$6$};
\node (z) at (0.5-2.5,2) [fill = white]  {$6$};
\node (z) at (0.5+2.5,2) [fill = white]  {$6$};

\node (z) at (-0.1,5) [fill = white]  {$7$};
\node (z) at (-0.1,2.5) [fill = white]  {$7$};
\node (z) at (-0.1-2.5,2.5) [fill = white]  {$7$};
\node (z) at (-0.1+2.5,2.5) [fill = white]  {$7$};

\node (z) at (0.75-2.95,2.55) [fill = white]  {$1$};

\node (z) at (0.75-0.45,2.45) [fill = white]  {$2$};

\draw (0,5) circle [thick, radius=1];
\node (t0) at (-0.25,5.5) [draw = black] {};
\node (t1) at (-0.6,5.2) [draw = black] {};
\node (t2) at (-0.6,4.8) [draw = black]  {};
\node (t3) at (-0.25,4.5) [draw = black]  {};
\node (t0r) at (0.25,5.5) [draw = black] {};
\node (t1r) at (0.6,5.2) [draw = black] {};
\node (t2r) at (0.6
,4.8) [draw = black]  {};
\node (t3r) at (0.25,4.5) [draw = black]  {};

\draw (0,2.5) circle [thick, radius=1];
\node (u0) at (-0.25,5.5-2.5) [draw = black] {};
\node (u1) at (-0.6,5.2-2.5) [draw = black] {};
\node (u2) at (-0.6,4.8-2.5) [draw = black]  {};
\node (u3) at (-0.25,4.5-2.5) [draw = black]  {};
\node (u0r) at (0.25,5.5-2.5) [draw = black] {};
\node (u1r) at (0.6,5.2-2.5) [draw = black] {};
\node (u2r) at (0.6
,4.8-2.5) [draw = black]  {};
\node (u3r) at (0.25,4.5-2.5) [draw = black]  {};

\draw (0-2.5,2.5) circle [thick, radius=1];
\node (v0) at (-0.25-2.5,5.5-2.5) [draw = black] {};
\node (v1) at (-0.6-2.5,5.2-2.5) [draw = black] {};
\node (v2) at (-0.6-2.5,4.8-2.5) [draw = black]  {};
\node (v3) at (-0.25-2.5,4.5-2.5) [draw = black]  {};
\node (v0r) at (0.25-2.5,5.5-2.5) [draw = black] {};
\node (v1r) at (0.6-2.5,5.2-2.5) [draw = black] {};
\node (v2r) at (0.6-2.5
,4.8-2.5) [draw = black]  {};
\node (v3r) at (0.25-2.5,4.5-2.5) [draw = black]  {};

\draw (0+2.5,2.5) circle [thick, radius=1];
\node (w0) at (-0.25+2.5,5.5-2.5) [draw = black] {};
\node (w1) at (-0.6+2.5,5.2-2.5) [draw = black] {};
\node (w2) at (-0.6+2.5,4.8-2.5) [draw = black]  {};
\node (w3) at (-0.25+2.5,4.5-2.5) [draw = black]  {};
\node (w0r) at (0.25+2.5,5.5-2.5) [draw = black] {};
\node (w1r) at (0.6+2.5,5.2-2.5) [draw = black] {};
\node (w2r) at (0.6+2.5
,4.8-2.5) [draw = black]  {};
\node (w3r) at (0.25+2.5,4.5-2.5) [draw = black]  {};

%\draw[yellow] [line width=0.5mm] (-0.95,0.3) -- (-0.6,0.17);
%\node (z) at (-0.95,0.3) [draw = black]  {};

%\draw[blue] [line width=0.5mm] (-0.75,0.66) -- (-0.47,0.43);
%\node (z) at (-0.75,0.66) [draw = black]  {};

%\filldraw (0,1) circle (1.8pt);
%\filldraw (0.955,0.3) circle (1.8pt);
%\filldraw (0.866,-0.5) circle (1.8pt);
%\filldraw (-0.7071,-0.7071) circle (1.8pt);
%\filldraw (-0.866,0.5) circle (1.8pt);
 \draw[dashed] (t0) edge  (t3);
 \draw[dashed] (u1) edge  (u2);
 \draw[dashed] (u1r) edge  (u3r);

 \draw[dashed] (v0) edge  (v1);
  \draw[dashed] (v0r) edge  (v2r);
 \draw[dashed] (w1) edge  (w2);
 \draw[dashed] (u2) edge  (u3);
  \draw[dashed] (w2) edge  (w3);
%\draw [dashed] (u4) to [out=70,in=190,looseness = 1] (u5);

 \draw [ultra thick] (w1) edge  (w0);
 \draw [ultra thick] (w2) edge  (w1);
 \draw [ultra thick] (u2) edge  (u3);
 \draw [ultra thick] (u2r) edge  (u3r);
 \draw [ultra thick] (u1r) edge  (u2r);
 \draw [ultra thick] (u1) edge  (u0);
  \draw [ultra thick] (v1) edge  (v2);
    \draw [ultra thick] (v0) edge  (v3);
      \draw [ultra thick] (u0) edge  (u3);
 
 \draw[ultra thick] (t1) edge  (t0);
 \draw[ultra thick] (t1) edge  (t2);
 \draw[ultra thick] (t2) edge  (t3);
  \draw[ultra thick] (w0) edge  (w3);

   \draw[ultra thick] (v0r) edge  (v1r);
   \draw[ultra thick] (v2r) edge  (v1r);
 \draw[ultra thick] (v2) edge  (v3);

\draw (0,4) -- (0,3.5);
\draw (-0.7071,5 -0.7071) -- (-2.5+0.7071,2.5+0.7071);
\draw (0.7071,5 -0.7071) -- (2.5-0.7071,2.5+0.7071);

\draw (0-0.5,2.5-0.866) -- (0-0.5*1.5,2.5-0.866*1.5);
\draw (0+0.5,2.5-0.866) -- (0+0.5*1.5,2.5-0.866*1.5);
\draw (0,1.5) -- (0, 1);
\draw (0-0.25882,2.5-0.96593) -- (0-0.25882*1.5,2.5-0.96593*1.5);
\draw (0+0.25882,2.5-0.96593) -- (0+0.25882*1.5,2.5-0.96593*1.5);

\draw (-2.5-0.5,2.5-0.866) -- (-2.5-0.5*1.5,2.5-0.866*1.5);
\draw (-2.5+0.5,2.5-0.866) -- (-2.5+0.5*1.5,2.5-0.866*1.5);
\draw (-2.5,1.5) -- (-2.5, 1);
\draw (-2.5-0.25882,2.5-0.96593) -- (-2.5-0.25882*1.5,2.5-0.96593*1.5);
\draw (-2.5+0.25882,2.5-0.96593) -- (-2.5+0.25882*1.5,2.5-0.96593*1.5);

\draw (2.5-0.5,2.5-0.866) -- (2.5-0.5*1.5,2.5-0.866*1.5);
\draw (2.5+0.5,2.5-0.866) -- (2.5+0.5*1.5,2.5-0.866*1.5);
\draw (2.5,1.5) -- (2.5, 1);

\foreach \from/\to in {t0/t1,t1/t2,t2/t3,t0r/t1r,t1r/t2r,t2r/t3r,
u0/u1,u2/u3,u0r/u1r,u1r/u2r,u2r/u3r,u0/u3,
v0/v3,v1/v2,v2/v3,v0r/v1r,v1r/v2r,v2r/v3r,v0/v3,
w0/w1,w1/w2,w0r/w1r,w1r/w2r,w2r/w3r,w0/w3}
    \draw (\from) -- (\to);
\end{tikzpicture} 
\caption{The figure shows the tree $\mathcal T$ as constructed when $F^*$ is a rainbow forest shown at the top of the figure, with $\phi(E(F^*)) = \{1,2,3,4,5,6\}$, and with $\sigma = 7$. The root of tree contains the pair $(F^*, 7)$. The children of the root of $\mathcal T$ are computed by considering the forests obtained after replacing an edge of $F^*$ with a new edge of color $7$. The leftmost child of the root of $\mathcal T$ is obtained by replacing the edge of color $1$ in $F^*$. This child contains a pair $(F',1)$, where $F'$ is the forest shown in the figure. The children of $(F',1)$ are computed similarly, by considering the forests obtained by replacing an edge of $F'$ by an edge of color $1$. The children of the other nodes are computed similarly. In each depicted node $\nu = (F, j)$, the dashed edges represent edges in $G_p \cup E(F^*)$ of the color $j$ missed by $F$, and the bolded edges represent edges of $F$ that may be replaced to produce children of $\nu$.
}
\label{figMoneyTree}
\end{figure}
We show an example of a tree $\mathcal T$ constructed by this process in Figure \ref{figMoneyTree}. We note that with this construction, a pair $(F,j)$ that appears once in $\mathcal T$ will appear infinitely often in $\mathcal T$. This is not a concern for us.

We argue that each pair $(F,j)$ stored by $\mathcal T$ satisfies the properties of Lemma \ref{lemmaMany}. The pair $(F^*,\sigma)$ clearly satisfies the properties of Lemma \ref{lemmaMany}. Next, suppose that some pair $(F,j)$ satisfies the properties of Lemma \ref{lemmaMany}; we show that the children of $(F,j)$ also satisfy these properties.
In more detail, let $e \in  E_j$, and let $r \in E(F)$ be replaced by $e$. We show that the pair $(F + e - r, \phi(r))$ satisfies the properties listed in the lemma.
First, since $e$ is replaceable by $r$, $F+e-r$ is a forest, and clearly the components of $F+e-r$ partitons $V(G)$ in the same way as those of $F$. Next, as $j \not \in \phi(E(F))$, $F+e-r$ is rainbow colored. Finally, $\phi(r)$ is the only color missed by $\phi(F+e-r)$. Hence, every node $(F,j) \in \mathcal T$ satisfies the properties listed in the lemma.

We would like to show that $\mathcal{T}$ stores a large number of distinct colors in its nodes, as this will ultimately allow us to show that the set $J$ defined in the lemma is large. In order to show that $\mathcal{T}$ stores a large number of distinct colors, we will compute a finite subtree $\mathcal{T}_f \subseteq \mathcal{T}$, each of whose nodes stores a distinct color, and we will show that $\mathcal{T}_f$ contains many nodes.

Throughout the proof, we write $\omega(n) = \frac{1}{3} \log \log \log n$.
We compute $\mathcal{T}_f$ in two parts. The first of these parts is the following algorithm, which will output a finite subtree $\mathcal{T}_f \subseteq \mathcal{T}$ containing at least $\frac{n}{\omega(n) \log n}$ nodes, each of whose nodes stores a distinct color. 
\begin{algorithm}[H]
\caption{Construct $\mathcal T_f$ (I)}\label{alg:cap}
\begin{algorithmic}[1]
\State Let $\mathcal T_f$ have a root $(F^*, \sigma)$.
\For{$h \geq 0$}
\State Let $\mathcal L_h$ be the set of vertices of $\mathcal{T}_f$ at height $h$.
\State \label{stepReplace} 
\multiline{For each pair $(F,j) \in \mathcal{L}_h$, add to $\mathcal{T}_f$ all children of $(F,j)$ in $\mathcal T$ whose color has not yet appeared in $\mathcal{T}_f$, adding at most one child of each color.}
\EndFor
\end{algorithmic}
\end{algorithm}

As $\mathcal T_f$ can contain at most $n - 1$ distinct colors, the Construct $\mathcal T_f$ (I) procedure must terminate. We first would like to show that $\mathcal T_f$ grows at a predictable rate and eventually contains at least $\frac{ n}{\omega(n) \log n}$ nodes. To this end, let $M$ be the maximum integer value for which all values $h < M$ satisfy $|\mathcal{L}_h |\leq \frac{ n}{\omega(n) \log n}$. If no such maximum exists, we let $M = \infty$. We define $t_{-1} = - 1 +\frac{2 \omega(n)}{\log \log n} $.
\begin{claim}
\label{claimF}
For $0 \leq h \leq M$, 
%after $h$ iterations of Step (4) of the algorithm above, 
the number of elements in $\mathcal L_h$ is a value $(\log n)^{t_h}$, where $t_h \geq t_{h-1} + 1 - \frac{2\omega(n)}{\log \log n}$.
\end{claim}
\begin{proof}[Proof of Claim \ref{claimF}]
We prove this claim by induction. When $h = 0$, $\mathcal{L}_h$ consists of the single element $(F^*, \sigma)$. Therefore, the number of elements in $\mathcal{L}_0$ is $1 = (\log n)^0$; hence, $t_0 = 0 = t_{-1} + 1 - \frac{2\omega(n)}{\log \log n}$.

Next, we show that if the lemma holds for values up to some $h < M$, then the lemma also holds for $h + 1$. In other words, we assume that $\mathcal L_h$ contains $(\log n)^{t_h}$ pairs $(F,j)$, and we would like to show that $\mathcal L_{h+1}$ contains at least $(\log n)^{t_h + 1 - \frac{2 \omega(n)}{\log \log n}}$ pairs. When $\mathcal L_{h+1}$ is constructed, for each pair $(F,j) \in \mathcal L_h$, we compute all children $(F',j')$ of $(F,j)$ for which $j'$ has not yet appeared as a color in $\mathcal{T}_f$. We would like to estimate the total number of these child nodes $(F',j')$, as these child nodes make up $\mathcal L_{h+1}$.

In order to estimate the number of nodes in $\mathcal L_{h+1}$, we will need to count the total number of colors among edges in forests of $\mathcal L_h$ that are replaced during the iteration of Line \ref{stepReplace} that generates $\mathcal L_{h+1}$. In order to make this estimate, 
consider a pair $(F,j) \in \mathcal L_h$. Every child of $(F,j)$ is obtained by choosing an edge $e \in E(G_j)$ and using $e$ to replace some edge $r \in E(F)$. If $\phi(r)$ has not yet appeared in a node of $\mathcal T_f$, then $\phi(r)$ will appear in a child node of $(F,j)$. Furthermore, we know that every edge of $F$ that also belongs to $F^*$ has a distinct color; therefore, if we can show that the number of replaced edges in $E(F) \cap E(F^*)$ is much larger than the number of colors already appearing in $\mathcal T_f$, then this will show that $(F,j)$ will have many children in $\mathcal T_f$. In fact, we will not consider single nodes $(F,j)$, but rather all nodes of $\mathcal L_h$ at the same time, but our process for estimating the overall number of children nodes that make up $\mathcal L_{h+1}$ will be just as we have described here.

We now rigorously compute a lower bound for the number of nodes in $\mathcal L_{h+1}$. For a pair $(F,j) \in \mathcal{L}_h$, let $R_{(F,j)} \subseteq E(F)$ denote the set of edges in $E(F)$ that are replaceable by the edge set $E_j \subseteq E(G_p)$, which consists of those edges in $G_p$ of color $j$.
Then, let 
$$R = \bigcup_{(F,j) \in \mathcal{L}_h} R_{(F,j)}.$$ 
We also let
$$A_h = \bigcup_{(F,j) \in \mathcal{L}_h} E_j,$$
and by Lemma \ref{lemmaHitV}, 
$$|V(A_h)| \geq |\mathcal L_h| \cdot \frac{\log n}{\exp(\omega(n))} = (\log n)^{t_h + 1 - \frac{\omega(n)}{\log \log n}}.$$
Now, recall that we have assumed that there is no rainbow forest on $G_p \cup E(F^*)$ with $|E(F^*)| + 1$ edges, and therefore, for each pair $(F,j)$ that we consider, every edge of $E_j$ must create a cycle when added to $F$.
Therefore, for each $(F,j) \in \mathcal{L}_h$, $R_{(T,j)}$ is an edge cover of $V(E_j)$; that is, every vertex incident to an edge of $E_j$ must also be incident to an edge of $R_{(F,j)}$. Therefore,
 ${R}$ is an edge cover of $V(A_h)$. As a single edge is only incident with two vertices, it must follow that 
$$|R| \geq \frac{1}{2}|V(A_h)| \geq \frac{1}{2} (\log n)^{t_h + 1 - \frac{\omega(n)}{\log \log n}}.$$
 
We write $R^* =  R \cap E(F^*)$.
As each color appearing in $E(F^*)$ is distinct, we observe that each edge $r \in R^*$ has a distinct color. Furthermore, for any node $(F,j) \in \mathcal T_f$ and any child $(F',j')$, $F$ and $F'$ differ by at most one edge. We let $$Q = \left | \bigcup_{i = 0}^{h} \mathcal L_i \right |,$$
and we see that immediately before $\mathcal L_{h+1}$ is constructed, 
$$\left | \bigcup_{(F,j) \in \mathcal T_f} E(F) \setminus E(F^*) \right | \leq Q.$$
 Therefore, $ |R^*| \geq |R| - Q $. Furthermore, the number of distinct colors already appearing in $ \bigcup_{i = 0}^{h} \mathcal L_i $ is equal to $Q$. Therefore, as each edge $ R^*$ whose color does not already appear in $ \bigcup_{i = 0}^{h} \mathcal L_i $ gives a new node in $\mathcal L_{h+1}$, the number of new nodes in $\mathcal L_{h+1}$ satisfies
$$|\mathcal L_{h+1}| \geq | R| - 2 Q \geq \frac{1}{2} (\log n)^{t_h + 1 - \frac{\omega(n)}{\log \log n}} - 2 Q.$$
 Hence, to finish the induction step, we only need to show that $Q$ is not too large. By the induction hypothesis,

\begin{eqnarray*}
Q = \sum_{i = 0}^{h}|\mathcal L_i| &= &(\log n)^{t_h }+ (\log n)^{t_{h-1} } + \dots + (\log n)^{t_{0} } \\
 & <&  (\log n)^{t_{h}} \left ( 1 + (\log n)^{-(1 - \frac{2\omega(n)}{\log \log n})}  + (\log n)^{-2(1 - \frac{2\omega(n)}{\log \log n})} + \dots \right )\\
 &  <& 2(\log n)^{t_h} .
\end{eqnarray*}

Therefore, the number of nodes in $\mathcal L_{h+1}$ is at least $\frac{1}{2} (\log n)^{t_h + 1 - \frac{\omega(n)}{\log \log n}} - 4 (\log n)^{t_h} > (\log n)^{t_h + 1 - \frac{2\omega(n)}{\log \log n}}$, and induction is complete. This proves Claim \ref{claimF}.
\end{proof}

By Claim \ref{claimF}, for some value $M < \log n$, $\mathcal{L}_{M}$ contains at least $ \frac{n}{ \omega(n) \log n}$ elements. Now, we may extend $\mathcal T_f$ so that it contains $\frac{n}{\log \log n}$ nodes, using the following procedure:
\begin{algorithm}[H]
\caption{Construct $\mathcal T_f$ (II)}\label{alg:cap}
\begin{algorithmic}[1]
\item Delete nodes from $\mathcal L_M$ until $\mathcal L_M$ has exactly $\lfloor \frac{n}{\omega(n) \log n } \rfloor $ nodes.
\item For each pair $(F,j) \in \mathcal{L}_M$, add to $\mathcal{T}_f$ all children of $(F,j)$ in $\mathcal T$ whose color has not yet appeared in $\mathcal{T}_f$, adding at most one child of each color. Call this new set of nodes $\mathcal L_{M+1}$.
\end{algorithmic}
\end{algorithm}
By the argument used in Claim \ref{claimF}, $\mathcal L_{M+1}$ must contain at least
$$ \left \lfloor \frac{n}{\omega(n) \log n } \right \rfloor \cdot (\log n)^{1 - \frac{2\omega(n)}{\log \log n}} > \frac{n}{ 2 \omega(n) } \cdot e^{-2 \omega(n)} > n \cdot e^{-3 \omega(n)} = \frac{n}{\log \log n}$$ nodes. As each node of $\mathcal T_f$ must contribute a distinct element to $J$, it follows that $|J| \geq \frac{n}{\log \log n}$. This completes the proof of Lemma \ref{lemmaMany}.
    \end{proof}

\subsection{Adding edges of the remaining colors:} 
Recall that we have a rainbow forest $F_0$ on $G_p$ with at least $n - \log^3 n$ edges.
We will finally show that with the help of our edge replacement technique, we may successively build larger rainbow forests $F^*$ until we have a rainbow spanning tree on $G'$. We iterate the following procedure, which builds a rainbow spanning tree on $G'$. 
 %\\ \\
%\textbf{Connect Forest Components:}
\begin{algorithm}[H]
\caption{Connect Forest Components}\label{alg:cfc}
\begin{algorithmic}[1]
%\item Let $\phi_0 = \tau$ denote the rainbow coloring function associated with $F_0$.
\State $F^* \gets F_0$
%\State $i \gets 1$
\State Set $t$ equal to the number of components in $F^*$.
\While{$t > 1$}
\label{stepGo} \If{a rainbow forest $F$ exists in $G_p \cup E(F^*)$ with $|E(F^*)| + 1$ edges}
\State $F^* \gets F$,
%\State Set $t$ to be the number of components in $F_{i-1}$.
%\State \label{stepTerminate} If $t=1$, then terminate. 
\Else
\State Choose some color $\sigma \in [n-1] \setminus \phi(E(F))$.% missed by $F_{i-1}$.
\State \label{stepJ} 
\multiline{Let $J \subseteq \phi(E(F^*)) \cup \{\sigma\}$ be the set of values $j$ 
for which there exists a rainbow forest $F$ on $G_p \cup E(F^*)$  
satisfying:}
\begin{itemize}
\setlength{\itemindent}{1.2cm}
\item $\phi(F) = \phi(E(F^*)) \cup \{\sigma \} \setminus \{j\}$, and
\item The components of $F$ partition $V(G)$ in the same way as those of $F^*$.
\end{itemize}
\State \label{stepSparseColor} 
Choose a value $j \in J$ for which there exists an edge $e \in E(G_{s_t})$ satisfying:
\begin{itemize}
\setlength{\itemindent}{1.2cm}
\item $e$ does not belong to any graph appearing before $G_{s_t}$ in the ordering of $\mathcal G$,
\item $\phi(e) = j$,
\item $e$ has endpoints in two distinct components of $F$.
\end{itemize}
\State 
\multiline{Let $F$ be a rainbow forest whose components partition $V(G)$ in the same way as those of $F^*$, and such that $\phi(F) = \phi(F^*) \cup \{\sigma\} \setminus \{j\}$.}
%\item Let $F_i$ have a rainbow coloring function $\phi_i$ for which $\phi_i(E(F_i)) = \phi_{i-1}(E(F_{i-1})) \cup \{\sigma \} \setminus \{j\}$.
\State $F^* \gets F + e$
%\item \label{stepLast} Set $i \leftarrow i+1$, and Goto (\ref{stepGo}).
\EndIf
\State $t \gets t - 1$ \Comment{Set $t$ equal to the number of components of $F^*$.}
\EndWhile
\end{algorithmic}
\end{algorithm}

Assuming that the Connect Forest Components procedure always iterates successfully, the procedure will produce increasingly large rainbow forests $F^*$ on $G'$. 
Furthermore, when the procedure finally terminates, we will have a rainbow spanning tree in $G'$.
Therefore, it remains only to show that the procedure a.a.s.~never fails to execute any step. 
We observe that since $F_0$ has at least $n - \lfloor \log^3 n\rfloor$ edges, the value of $t$ in each iteration of the Connect Forest Components procedure is at most $\lfloor \log^3 n \rfloor - 1$, so the graph $G_{s_t}$ called in each iteration is always well defined.
Hence, Line \ref{stepSparseColor} is the only step that might fail, so we only need to check that Line \ref{stepSparseColor} is a.a.s.~executed successfully each time that it is called.

By Lemma \ref{lemmaMany}, the set $J$ produced in Line \ref{stepJ} always has at least $\frac{n}{\log \log n}$ elements.
 Also, by Lemma \ref{lemmaStraddle}, in each iteration, $\binom{V(G)}{2}$ contains at least $\frac{1}{4} \lambda t $ edges with endpoints in distinct components of $F$
 that do not belong to any graph appearing before $G_{s_t}$ in $\mathcal G$. Furthermore, $G_{s_t}$ is never observed before being used in Line \ref{stepSparseColor}, and therefore the edges in $G_{s_t}$ are independent of any previous observations. 
 Therefore, the probability that Line \ref{stepSparseColor} fails on a given iteration is at most
$$\left (1 - \frac{s_t}{\log \log n} \right )^{ \frac{1}{4} \lambda t } < \exp \left ( - s_t \cdot \frac{\lambda  t}{4 \log \log n} \right ) = \exp \left (- \Omega \left ( \frac{\sqrt{\log n}}{ \log \log n} \right)  \right ) = o \left ( (\log n)^{-3} \right ).$$ 
As at most $(\log n)^3$ edges need to be added to $F_0$ to obtain a rainbow spanning tree, Line \ref{stepSparseColor} is called at most $(\log n)^3$ times. Thus, it a.a.s.~holds that the Construct Forest Components procedure will successfully iterate until producing a rainbow spanning tree on $G'$. This completes the proof of Theorem \ref{thmTree}. \qed

\section{Conclusion}
A \emph{matroid} is a set $E$ of elements along with a nonempty  family $\mathcal B \subseteq 2^{E}$ of subsets of $E$ called \emph{bases}, such that
if $A,B \in \mathcal B$ 
and $a \in A \setminus B$, then there exists an element $b \in B$ for which $B \setminus \{b\} \cup \{a\} \in \mathcal B$.
This condition on $\mathcal B$ is often called the \emph{basis exchange property}.
A connected graph $G$ is an example of a matroid, in which case $E = E(G)$, and $\mathcal B$ consists of 
the family $\mathcal T(G)$ spanning trees of $G$. In this case, the basis exchange property describes the fact that for any spanning tree $T$ of $G$ and any edge $e \in E(G) \setminus E(T)$, $T$ has at least one edge that is replaced by $e$. 
%A \emph{cocircuit} of a matroid $(E,\mathcal B)$ is a minimal subset $C \subseteq E$ that intersects with every basis in $\mathcal B$. 
%If $G$ is a $\lambda$-edge-connected graph, then every cocircuit in the corresponding matroid $(E(G), \mathcal T(G))$ corresponds to an edge-cut of $G$ and hence contains at least $\lambda$ elements. 
Other graph properties such as degree and edge-connectivity have analagous notions in matroid theory.
See \cite{matroids} for a detailed introduction to matroids.

Our method for proving Theorem \ref{thmTree} relies heavily on replacing edges in forests of a random subgraph of a graph. Given that edge replacement has a natural matroid equivalent, it is natural to ask if our techniques can be generalized to the setting of matroids, as follows. Suppose that $(E,\mathcal B)$ is a matroid, and suppose that $E' \subseteq E$ is a randomly chosen subset of $E$ whose elements are randomly colored. Under what conditions does the set $E'$ a.a.s.~contain a rainbow-colored basis from $\mathcal B$?

\section{Acknowledgments}
I am grateful to the graph theory group of Simon Fraser University for listening to and commenting on several presentations of the main ideas of this paper. In particular, I am grateful to Kevin Halasz, Bojan Mohar, Ladislav Stacho, and Luis Goddyn for helpful discussions. I am also grateful to Joshua Erde for reading an earlier draft of this paper and for pointing out important background results on this topic.

\raggedright
\bibliographystyle{abbrv}
\bibliography{bib}

\end{document}